\documentclass[11 pt]{amsart}
\usepackage{amscd,amsfonts,amssymb,amsmath}
\usepackage{hyperref}

\usepackage{tikz}
\usepackage[margin=3.7 cm]{geometry}

\newtheorem{theorem}{Theorem}[section]
\newtheorem{cor}[theorem]{Corollary}
\newtheorem{lemma}[theorem]{Lemma}
\newtheorem{prop}[theorem]{Proposition}

\theoremstyle{definition}

\newtheorem{remark}[theorem]{Remark}
\numberwithin{equation}{subsection}
\theoremstyle{plain}
\newtheorem*{ack}{Acknowledgement}

\newtheorem{question}{Question}

\newcommand{\Aut}{\operatorname{Aut}}

\newcommand{\Mod}{\operatorname{Mod}}

\newcommand{\C}{\operatorname{C}}

\newcommand{\Inn}{\operatorname{Inn}}

\numberwithin{equation}{section}

\begin{document}
\title{Automorphisms of pure braid Groups}
\author[V. G. Bardakov]{Valeriy G. Bardakov}
\author[M. V. Neshchadim]{Mikhail V. Neshchadim}
\author[M. Singh]{Mahender Singh}

\date{\today}
\address{Sobolev Institute of Mathematics and Novosibirsk State University, Novosibirsk 630090, Russia.}
\address{Laboratory of Quantum Topology, Chelyabinsk State University, Brat'ev Kashirinykh street 129, Chelyabinsk 454001, Russia.}
\address{Novosibirsk State Agrarian University, Dobrolyubova street, 160, Novosibirsk, 630039, Russia.}
\email{bardakov@math.nsc.ru}

\address{Sobolev Institute of Mathematics and Novosibirsk State University, Novosibirsk 630090, Russia.}
\email{neshch@math.nsc.ru}

\address{Indian Institute of Science Education and Research (IISER) Mohali, Sector 81,  S. A. S. Nagar, P. O. Manauli, Punjab 140306, India.}
\email{mahender@iisermohali.ac.in}


\begin{abstract}
In this paper, we investigate the structure of the automorphism groups of pure braid groups. We prove that, for $n>3$,  $\Aut(P_n)$ is generated by the subgroup  $\Aut_c(P_n)$ of central automorphisms of $P_n$, the subgroup $\Aut(B_n)$ of restrictions of automorphisms of $B_n$ on $P_n$ and one extra automorphism $w_n$. We also investigate the lifting and extension problem for automorphisms of some well-known exact sequences arising from braid groups, and prove that that answers are negative in most cases. Specifically, we prove that no non-trivial central automorphism of $P_n$ can be extended to an automorphism of $B_n$.
\end{abstract}
\maketitle

\section{Introduction}
Let $B_n$ be the Artin braid group on $n$-strings, and  $\{\sigma_1,\dots,\sigma_{n-1}\}$ the set of standard generators of $B_n$.  In \cite{Artin1}, Artin mentioned the problem of determining all automorphisms of the braid groups. Subsequently, in \cite{Artin2}, he determined all representations of $B_n$ by transitive permutation groups in $n$ letters. Since then the subject has attracted a lot of attention. In \cite{Dyer-Gross}, Dyer and Grossman completely determined the automorphism groups of braid groups. More precisely, they proved that
 $$\Aut(B_n) \cong \Inn(B_n) \rtimes \mathbb{Z}_2,$$
where $\Inn(B_n) \cong B_n/Z(B_n)$ is the group of inner automorphisms and $\mathbb{Z}_2= \langle \tau \rangle$ with $\tau(\sigma_i)=\sigma_i^{-1}$ for all $1 \le i \le n-1$.

The pure braid group $P_n$ is the kernel of the natural homomorphism $B_n \to S_n$, from the braid group to the symmetric group. Once the structure of $\Aut(B_n)$ is revealed, it is natural to investigate the structure of $\Aut(P_n)$. It is well known that $P_n \cong Z(P_n) \times \overline{P}_n$, where $\overline{P}_n=P_n/Z(P_n)$.  Using this decomposition, Bell and Margalit \cite[Theorem 8]{Bell} proved that,  for $n \ge 4$, 
$$\Aut(P_n)\cong \Aut_c(P_n) \rtimes \Aut(\overline{P}_n),$$ 
where $\Aut(\overline{P}_n)\cong \Mod(\mathbb{S}_{n+1})$. Here $\Aut_c(P_n)$ is the group of central automorphisms of $P_n$ and
$\Mod(\mathbb{S}_{n+1})$ is the extended mapping class group of the 2-sphere
with $n+1$ punctures.
Hence $\Aut(P_n)\cong \Aut_c(P_n) \rtimes \Mod(\mathbb{S}_{n+1})$  for $n \ge 4$. This result was used by Cohen \cite{Cohen} to give an explicit presentation of the automorphism group $\Aut(P_n)$ for $n \geq 3$. Notice that from this point of view the result of Dyer and Grossman  has the form
$$
\Aut(B_n)\cong  \Mod(\mathbb{D}_{n}),
$$
where $\mathbb{D}_n$ is the disc with $n$ punctures.

Since the group of pure braids $P_n$ is the direct product of the center
$Z(P_n)$ and the group $\overline{P}_n$, it is not difficult to determine the
group of central automorphisms $\Aut_c(P_n)$. On the other hand, $P_n$ is
a characteristic subgroup of $B_n$, and hence each automorphism of $B_n$ induces
an automorphism of $P_n$. This leads to the following question:
\begin{question}
Is $\Aut(P_n)$ generated by $\Aut_c(P_n)$ and $\Aut(B_n)$?
\end{question}
After recalling some preliminary results in Section \ref{sec2}, we answer the preceding question in Sections \ref{sec3} and \ref{sec4} of the paper. The answer is evidently affirmative for $n=2$, and we prove that it is so for $n=3$ as well (Proposition \ref{main-result-p3}). For $n > 3$, we prove that $\Aut(P_n)$ is generated by $\Aut_c(P_n)$, $\Aut(B_n)$ and one extra automorphism $w_n$ (Theorem \ref{main-result-pn}).

The last two sections of the paper are dedicated to the lifting and extension of automorphisms in certain short exact sequences arising from braid groups. A general formulation is the following:

\begin{question}
Let  $1 \to K \to G \to H \to 1$ be a short exact sequence of groups. Given $\phi \in \Aut(H)$, does there exists an automorphism of $G$ which induces $\phi?$ Analogously,  given $\psi \in \Aut(K)$, does there exists an automorphism of $G$ whose restriction to $K$ is $\psi?$
\end{question}

For finite groups, this problem has been investigated in \cite{Jin, PSY, Robinson} using cohomological methods and the fundamental exact sequence of Wells \cite{Wells}. In Section \ref{sec5}, we prove that the answers to both the questions are negative (Theorems \ref{not liftable auto} and \ref{not-extendable-auto}) for the extension $1 \to U_4 \to P_4 \to P_3 \to 1$. Finally, in Section \ref{sec6}, we consider the extension $1 \to P_n \to B_n \to S_n \to 1$. We prove that the non-inner automorphism of $S_6$ cannot be lifted to an automorphism of $B_6$ (Proposition \ref{lift-inner-sn}), and that no non-trivial element of $\Aut_c(P_n)$ can be extended to an automorphism of $B_n$ (Theorem \ref{main-result-sn}).
\bigskip

\section{Preliminary results}\label{sec2}
We use the standard notations. Given two elements $x,y$ of a group $G$, we write $y^x = x^{-1}yx$ and $[x, y] = x^{-1} y^{-1} x y$. Further, for $z \in G$, $\hat{z}$ denotes the inner automorphism of $G$ induced by $z$. An automorphism of a group $G$ is called {\it central} if it induces identity automorphism on the central quotient. The subgroup of all central automorphisms of $G$ is denoted by $\Aut_c(G)$.

Next we recall some basic facts on braid groups and refer the reader to \cite{Bir, M} for more details. The {\it braid group} $B_n$, $n \geq 2$, on $n$-strings is generated by the set
$$
\{\sigma_1, \sigma_2, \ldots, \sigma_{n-1} \},
$$
and is defined by relations

 \begin{eqnarray*}
\sigma_i \sigma_j &=& \sigma_j \sigma_i  ~~~\mbox{for}~|i - j| >1, \\
\sigma_i \sigma_{i+1} \sigma_i &=& \sigma_{i+1} \sigma_i \sigma_{i+1}  ~\mbox{for}~i = 1, 2, \ldots, n-2.
 \end{eqnarray*}

Note that $B_1$ is defined as the trivial group and $B_2$ turns out to be the infinite cyclic group. The subgroup of $B_n$ generated by the elements
$$
A_{i,j} = \sigma_{j-1} \sigma_{j-2} \ldots \sigma_{i+1} \sigma_{i}^2 \sigma_{i+1}^{-1} \ldots \sigma_{j-2}^{-1} \sigma_{j-1}^{-1}~~\textrm{where}~1 \leq i < j \leq n,
$$
is called the {\it pure braid group}, and denoted by $P_n$. Further, $P_n$ is defined by the relations
 \begin{eqnarray*}
A_{i,k} A_{ij} A_{k,j} &=& A_{k,j} A_{i,k} A_{ij},   \label{re2}\\
A_{l,j} A_{k,l} A_{k,j} &=& A_{k,j} A_{l,j} A_{k,l}  ~\mbox{for}~l < j, \label{re3}\\
(A_{k,l} A_{k,j} A_{k,l}^{-1}) A_{i,l}& =& A_{i,l} (A_{k,l} A_{k,j} A_{k,l}^{-1})  ~\mbox{for}~i < k < l < j, \label{re4}\\
 A_{k,j} A_{i,l}& = &A_{i,l} A_{kj}  ~\mbox{for}~k < i < l < j ~\mbox{or}~l < k. \label{re1}
\end{eqnarray*}
It is readily seen that $P_1$ is the trivial group, $P_2 \cong \mathbb{Z}$, and $P_3\cong F_2 \times \mathbb{Z}$.  In general, $P_n$ is characteristic in $B_n$, and the quotient $B_n / P_n$ is the symmetric group $S_n$. Further, the generators of $B_n$ act on the generators $A_{i,j} \in P_n$ by the following rules:
 \begin{eqnarray*}
 \sigma_k^{-\varepsilon} A_{i,j} \sigma_k^{\varepsilon} &=&  A_{i,j}  ~\mbox{for}~k \not= i-1, i, j-1, j, \label{c1}\\
\sigma_{i}^{-\varepsilon} A_{i,i+1} \sigma_{i}^{\varepsilon} &=&  A_{i,i+1},   \label{c2}\\
 \sigma_{i-1}^{-1} A_{i,j} \sigma_{i-1} &=&   A_{i-1,j},  \label{c3}\\
 \sigma_{i-1}A_{i,j}\sigma_{i-1}^{-1}&= &A_{i,j}^{-1}A_{i-1,j}A_{ij}, \\
 \sigma_{i}^{-1} A_{i,j} \sigma_{i} &= & A_{i+1,j} [A_{i,i+1}^{-1}, A_{i,j}^{-1}]  ~\mbox{for}~j \not= i+1 \label{c4}, \\
   \sigma_{i}A_{i,j}\sigma_{i}^{-1}&=& A_{i+1,j} ~\mbox{for}~j \not= i+1,  \\
 \sigma_{j-1}^{-1} A_{i,j} \sigma_{j-1} &=&  A_{i,j-1},   \label{c5}\\
   \sigma_{j-1}A_{i,j}\sigma_{j-1}^{-1}&=& A_{i,j-1} \left[A_{ij}^{-1},A_{j-1,j}^{-1}\right], \\
 \sigma_{j}^{-1} A_{i,j} \sigma_{j}& = & A_{i,j} A_{i,j+1} A_{i,j}^{-1},   \label{c6} \\
  \sigma_{j}A_{i,j}\sigma_{j}^{-1}&=& A_{i,j+1}\,\, \mbox{for} \,\, 1\leq i < j\neq n-1,
\end{eqnarray*}
where $\varepsilon = \pm 1$.
\medskip

For each $i = 2, \ldots, n$, consider the following subgroup 
$$
U_{i} = \langle A_{1,i}, A_{2,i}, \ldots, A_{i-1,i} \rangle,
$$
of $P_n$. It is known that each $U_i$ is a free subgroup of rank $i-1$.
One can rewrite the relations of $P_n$ as the following conjugation rules (for $\varepsilon = \pm 1$):
 \begin{eqnarray*}
A_{i,k}^{-\varepsilon} A_{k,j}  A_{i,k}^{\varepsilon} &=&   (A_{i,j} A_{k,j})^{\varepsilon} A_{k,j} (A_{i,j} A_{k,j})^{-\varepsilon},  \label{co1}\\
 A_{k,l}^{-\varepsilon} A_{k,j}  A_{k,l}^{\varepsilon} &=&
   (A_{k,j} A_{l,j})^{\varepsilon} A_{k,j} (A_{k,j} A_{l,j})^{-\varepsilon}  ~\mbox{for}~l < j, \label{co2}\\
 A_{i,l}^{-\varepsilon} A_{k,j}  A_{i,l}^{\varepsilon}& =&
   [A_{i,j}^{-\varepsilon}, A_{l,j}^{-\varepsilon}]^{\varepsilon} A_{k,j} [A_{i,j}^{-\varepsilon}, A_{l,j}^{-\varepsilon}]^{-\varepsilon}  ~\mbox{for}~i < k < l, \label{co3}\\
 A_{i,l}^{-\varepsilon} A_{k,j} A_{i,l}^{\varepsilon} &=&
A_{k,j}~  ~\mbox{for}~k < i < l < j ~\mbox{or}~  l < k. \label{co4}
 \end{eqnarray*}

It follows from these rules that $U_n$ is normal in $P_n$, and $P_n$  has the decomposition $P_n = U_n \rtimes P_{n-1}$. This gives rise to the split extension $$1 \to U_n \to P_n \to P_{n-1} \to 1.$$
By induction on $n$, it follows that $P_n$ is the semi-direct products of free groups as follows
$$
P_n \cong U_n \rtimes (U_{n-1} \rtimes (\cdots  \rtimes (U_3 \rtimes U_2)\cdots )).
$$

It is known that,  for $n \ge 2$, the center $Z(P_n)=\langle z_n \rangle$ is infinite cyclic generated by the full twist braid 
$$z_n= A_{1,2} (A_{1,3}A_{2,3})\cdots (A_{1,n} \cdots A_{n-1,n}).$$
 It follows that $P_n \cong Z(P_n) \times \overline{P}_n$, where $\overline{P}_n=P_n/Z(P_n)$ (see \cite{Bir}). More precisely,
$$
P_n \cong Z(P_n) \times \langle A_{1,3}, A_{2,3}, \dots, A_{1,n}, \dots, A_{n-1,n} \rangle,
$$
Using this decomposition, Bell and Margalit \cite[Theorem 8]{Bell} proved the following result.

\begin{theorem}
$\Aut(P_n)\cong \Aut_c(P_n) \rtimes \Aut(\overline{P}_n)$  for $n \ge 4$.
\end{theorem}

Next, we recall a presentation for $\Aut(P_n)$ from Cohen \cite{Cohen}. The  subgroup of central automorphisms $\Aut_c(P_n)$ consists of automorphisms of the form
$$
A_{i, j} \mapsto A_{i,j} z_n^{t_{ij}},
$$
where $t_{ij} \in \mathbb{Z}$ and $\sum t_{ij}=0$ or -2. We have $z_n \mapsto z_n$
in the former case, and $z_n \mapsto z_n^{-1}$ in the latter case.
This gives a surjection $$\Aut_c(P_n) \to \mathbb{Z}_2$$ with kernel
consisting of automorphisms  for which $\sum t_{ij}=0$. Since $P_n$
has $\binom{n}{2}$ generators, this kernel is free abelian of rank $N$, where $N=\binom{n}{2}-1$.
Further, the choice $t_{12}=-2$ and all other $t_{ij}=0$ gives a splitting $\mathbb{Z}_2 \to \Aut_c(P_n)$, and hence
$$\Aut_c(P_n) \cong \mathbb{Z}^N \rtimes \mathbb{Z}_2.$$
As in \cite[Equation (11)]{Cohen}, we note that this group is generated by the automorphisms
$\psi, \phi_{i,j}:P_n \to P_n$, $1 \leq i < j \leq n$, $\{i,j \} \neq \{1,2 \}$, where

$$
\psi : A_{p,q} \mapsto \left\{
\begin{array}{ll}
A_{1,2}z_n^{-2} &~\textrm{for}~p=1, q=2\\
A_{p,q} &~\textrm{otherwise}\\
\end{array} \right.
$$
and
$$
\phi_{i,j} : A_{p,q} \mapsto \left\{
\begin{array}{ll}
A_{1,2}z_n &~\textrm{for}~p=1, q=2\\
A_{i,j}z_n^{-1} &~\textrm{for}~p=i, q=j\\
A_{p,q} &~\textrm{otherwise}.\\
\end{array} \right.
$$
It can be easily checked that $\psi^2=1$ and $\psi \phi_{i,j} \psi= \phi_{i,j}^{-1}$.

Let $\mathbb{S}_{n+1}$ denote the 2-sphere with $n+1$ punctures, and
$\Mod (\mathbb{S}_{n+1})$ its extended mapping class group. For $n\ge 1$, $\Mod (\mathbb{S}_{n+1})$
has a presentation with generators $\{\omega_1,\dots,\omega_n, \varepsilon\}$ and relations

\begin{equation}
\begin{array}{ll} \label{mcg-gen-rel}
 \omega_i\omega_j=\omega_j\omega_i\ \text{for}\ |i-j|> 1,\\
\omega_i \omega_{i+1} \omega_i= \omega_{i+1} \omega_i \omega_{i+1}, \\
\omega_1\cdots \omega_{n-1}\omega_{n}^2\omega_{n-1}\cdots \omega_1=1,\\
(\omega_1 \omega_2 \cdots \omega_n)^{n+1}=1,\\
(\varepsilon \omega_i)^2=1,~\textrm{and}~ \varepsilon^2=1.
\end{array}
\end{equation}
We refer the reader to \cite[Theorem 4.5]{Bir} for details. By Bell and Margalit \cite{Bell}, $\Mod(\mathbb{S}_{n+1}) \cong \Aut(\overline{P}_n)$ and can be viewed as a subgroup of   $\Aut(P_n)$.
Further, the generators $\omega_1,\omega_2, \dots, \omega_n, \varepsilon$ act as automorphisms of $P_n$ in the following manner \cite[Equation (12)]{Cohen}:

$$
\omega_k \colon A_{i,j} \mapsto \left\{
\begin{array}{ll}
A_{i-1,j}&\text{if $k=i-1$}\\
A_{i+1,j}^{A_{i,i+1}}&\text{if $k=i<j-1$}\\
A_{i,j-1}&\text{if $k=j-1>i$}\\
A_{i,j+1}^{A_{j,j+1}}&\text{if $k=j$}\\
A_{i,j}&\text{otherwise},
\end{array} \right.
$$
for~ $1\le k \le n-1$~ and~$k \neq 2$;

$$
\omega_2 \colon A_{i,j} \mapsto \left\{
\begin{array}{ll}
A_{1,3}^{A_{2,3}}z_n &\text{if $i=1$, $j=2$}\\
A_{1,2}z_n^{-1}&\text{if $i=1$, $j=3$}\\
A_{3,j}^{A_{2,3}}&\text{if $i=2$, $j\ge 4$}\\
A_{2,j}&\text{if $i=3$}\\
A_{i,j}&\text{otherwise,}
\end{array} \right.
$$

$$
\omega_n \colon A_{i,j} \mapsto \left\{
\begin{array}{ll}
A_{i,j}&\text{if $j \neq n$}\\
(A_{1,n}A_{1,2}A_{1,3}\cdots A_{1,n-1})^{-1}z_n &\text{if $i=1$, $j=n$}\\
(A_{2,n}A_{1,2}A_{2,3}\cdots A_{2,n-1})^{-1}z_n &\text{if $i=2$, $j=n$}\\
(A_{i,n}A_{1,i}\cdots A_{i-1,i} A_{i,i+1}\cdots A_{i,n-1})^{-1}&\text{if $3\le i$, $j=n$,}
\end{array} \right.
$$
and 
$$
\varepsilon \colon A_{i,j} \mapsto \left\{
\begin{array}{ll}
A_{1,2}^{-1} z_n^2&\text{if $i=1$, $j=2$}\\
(A_{i+1,j} \cdots A_{j-1,j})^{-1} A_{i,j}^{-1} (A_{i+1,j} \cdots A_{j-1,j})&\text{otherwise.}
\end{array} \right.
$$
\bigskip

\section{Automorphism group of $P_n$ for $n \ge 4$}\label{sec3}

As we noted in the introduction $\Aut(B_n) = \langle \Inn(B_n), \tau \rangle$, and each automorphism of $B_n$ induces an automorphism of $P_n$.
 Let $t= \tau|_{P_n}$ and $s_i=\hat{\sigma_i}|_{P_n}$ for all $1 \le i \le n-1$, where $\hat{\sigma_i}$ is the inner automorphism of $B_n$ induced by $\sigma_i$.

\begin{lemma}
The automorphism $\tau$ induces the  automorphism $t$ of $P_n$ which acts on the generators by the rule
$$
t(A_{i,j})=(A_{i,j} A_{i+1,j}\ldots A_{j-1,j})^{-1} A_{i,j}^{-1} (A_{i,j} A_{i+1,j}\ldots A_{j-1,j})
$$
for all $1 \le i < j \le n$.
\end{lemma}

\begin{proof}
We use induction on $n$. Notice that it suffices to prove the equality
$$
t(A_{1,n})=(A_{1,n}A_{2,n}\ldots A_{n-1,n})^{-1} A_{1,n}^{-1}(A_{1,n}A_{2,n}\ldots A_{n-1,n}).
$$

For $n=3$ we have
\begin{eqnarray*}
t(A_{1,3}) &=& (\sigma_2 \sigma_1^2 \sigma_2^{-1})^{\tau}\\
& = & \sigma_2^{-1} \sigma_1^{-2} \sigma_2\\
& = & \sigma_2^{-2} \sigma_2 \sigma_1^{-2} \sigma_2^{-1} \sigma_2^2\\
& = & A_{2,3}^{-1}A_{1,3}^{-1}A_{2,3}\\
& = & (A_{1,3}A_{2,3})^{-1} A_{1,3}^{-1}(A_{1,3}A_{2,3}).
\end{eqnarray*}

Suppose that the formula holds for  $n$. Then
\begin{eqnarray*}
t(A_{1,n+1}) &=& t(\sigma_n A_{1,n}\sigma_n^{-1})\\
& = & \sigma_n^{-1} A_{1,n}^{\tau}\sigma_n\\
& = & \sigma_n^{-1} (A_{1,n}A_{2n}\ldots A_{n-1,n})^{-1}A_{1,n}^{-1}(A_{1,n}A_{2,n}\ldots A_{n-1,n})\sigma_n\\
& = & \sigma_n^{-2}\sigma_n (A_{1,n}A_{2,n}\ldots A_{n-1,n})^{-1}     A_{1,n}^{-1}(A_{1n}A_{2,n}\ldots A_{n-1,n})\sigma_n^{-1}\sigma_n^2\\
&  & (\textrm{using the formula of conjugation} ~ \sigma_n A_{i,n}\sigma_n^{-1}=A_{i,n+1})\\
& = & A_{n,n+1}^{-1} (A_{1,n+1}A_{2,n+1}\ldots A_{n-1,n+1})^{-1}     A_{1,n+1}^{-1}(A_{1,n+1}A_{2,n+1}\ldots A_{n-1,n+1})A_{n,n+1}\\
& = & (A_{1,n+1}A_{2,n+1}\ldots A_{n,n+1})^{-1}   A_{1,n+1}^{-1}(A_{1,n+1}A_{2,n+1}\ldots A_{n,n+1}).
\end{eqnarray*}

This proves the lemma.
\end{proof}

An immediate consequence is the following.

\begin{cor} The following equalities hold
$$
t \varepsilon(A_{1,2})=A_{1,2}z_n^{-2} ~~\textrm{and}\,\,
t \varepsilon(A_{i,j})=A_{ij},\,\,
 1\leq i < j\leq n, \,\, \left\{i,j \right\}\neq \left\{1,2 \right\},
$$
i.e. $t \varepsilon =\psi.$
\end{cor}

It is not difficult to see that the restriction homomorphism $\Aut(B_n) \to \Aut(P_n)$ is an embedding, enabling us to view  $\Aut(B_n)$ as a subgroup of $\Aut(P_n)$. Now, to determine $\Aut(P_n)$, introduce the following automorphism of $P_n$:
$$
w_n \colon A_{i,j} \mapsto \left\{
\begin{array}{ll}
A_{i,j}&\text{if $j \neq n$}\\
(A_{1,n}A_{1,2}A_{1,3}\cdots A_{1,n-1})^{-1} &\text{if $i=1$, $j=n$}\\
(A_{2,n}A_{1,2}A_{2,3}\cdots A_{2,n-1})^{-1} &\text{if $i=2$, $j=n$}\\
(A_{i,n}A_{1,i}\cdots A_{i-1,i} A_{i,i+1}\cdots A_{i,n-1})^{-1}&\text{if $3\le i$, $j=n$,}
\end{array} \right.
$$

Since the maps   $\omega_n$ and  $w_n$ differ by a central automorphism, it follows that $w_n$ also is an automorphism of the group $P_n$.

\begin{lemma}\label{wn not}
The automorphism $w_n$ does not lie in the subgroup $\langle \Aut_c(P_n), \Aut(B_n)\rangle$ for $n \geq 4$.
\end{lemma}

\begin{proof}
The automorphism $w_n$ induces an automorphism $\overline{w}_n$ of the abelianisation $P_n / P_n'$.
We prove that $\overline{w}_n$  does not lies in the image of
$\langle \Aut_c(P_n), \Aut(B_n)\rangle$ into $\Aut(P_n / P_n')$. Indeed, since
 $\Aut(B_n)=\left\langle \Inn( B_n),\,\, \tau  \right\rangle$,
the actions of elements of $\Aut(B_n)$ on the quotient  $P_n/P_n'$ are compositions of inversion of  $A_{ij}$ and permutations of these elements.
Hence, it is enough to show that, in the quotient $P_n/P_n'$, the set consisting of
$$
A_{i,j},\,\,1\leq i < j\leq n-1,
$$
$$
(A_{1,n}A_{1,2}A_{1,3}\ldots A_{1,n-1})^{-1},
$$
$$
(A_{2,n}A_{1,2}A_{2,3}\ldots A_{2,n-1})^{-1},
$$
$$
(A_{i,n}A_{1,i}\ldots A_{i-1,i}A_{i,i+1}\ldots A_{i,n-1})^{-1},  \,\, i \geq 3,
$$
cannot be obtained from the set
$$
A_{i,j},\,\,1\leq i < j\leq n
$$
using composition of the maps
$$
A_{i,j} \mapsto A_{k,l}^{\alpha_{ij}}z_n^{\beta_{ij}},
$$
where $\alpha_{ij}=\pm 1$, $\beta_{ij}\in \mathbb{Z}$, $\left\{i,j \right\}\mapsto \left\{k,l \right\}$
is the permutation of pairs of indexes $\left\{i,j \right\}$, $1\leq i < j\leq n$. Note that in this case $z_n \mapsto z_n^{\pm 1}$.
Let us take
$$
w_n(A_{1,n})=(A_{1,n}A_{1,2}A_{1,3}\ldots A_{1,n-1})^{-1}.
$$
If the transformation  $w_n$ modulo  $P_n'$
is a composition of maps of the specified type, then the word
$$
(A_{1,n}A_{1,2}A_{1,3}\ldots A_{1,n-1})^{-1}z_n^{\beta}
$$
with a suitable  $\beta\in \mathbb{Z}$ has the form  $A_{k,l}^{\pm 1}$, but it is not true. Hence $w_n$ does not lie in $\langle \Aut_c(P_n), \Aut(B_n)\rangle$.
\end{proof}

\begin{cor}
The subgroup $\left\langle \Aut_c( P_n),\,\,\Aut (B_n)  \right\rangle$
is not normal in the group $\Aut(P_n)$ for  $n\geq 4$.
\end{cor}

\begin{proof}
Set $H_n=\left\langle \Aut_c( P_n),\,\,\Aut (B_n)  \right\rangle$, and suppose that $H_n$ is normal in $\Aut(P_n)$.
Then the automorphism $w_n$ induces an automorphism of $\Aut(B_n)\cong H_n/ \Aut_c(P_n)$. By \cite[Theorem 22]{Dyer-Gross}, the group
$\Aut(B_n)$ is complete for $n\geq 4$, i.e. has trivial center and only inner automorphisms.
Since $\Aut( B_n)=\left\langle \Inn(B_n),\,\, \tau  \right\rangle$,
$\Aut(B_n)/\Inn (B_n) \cong \mathbb{Z}_2$ and $w_n \not\in H_n$ by Lemma \ref{wn not}, it follows that the product $w_n t$ must lies in the subgroup generated by inner automorphism group $\Inn (B_n)$ and central automorphism group $\Aut_c(P_n)$, which is a contradiction. Hence $H_n$ is not normal in $\Aut(P_n)$ for  $n\geq 4$.
\end{proof}


Let $G_n$ be the subgroup of $\Aut(P_n)$ consisting of those automorphisms which
are restrictions of  automorphisms of $B_n$ and  the automorphism $w_n$, i.e.
$$G_n= \langle t, s_1,\ldots, s_{n-1}, w_n \rangle .$$

\begin{lemma}\label{epsilon-lemma}
The automorphisms $\omega_1, \omega_2, \dots, \omega_n, \varepsilon$ lie in $\langle \Aut_c(P_n), G_n \rangle$.
\end{lemma}

\begin{proof}
 It follows that $\omega_k(z_n)=z_n$ for all $1\le k \le n$, and $\varepsilon(z_n)=z_n$.
 As noted by Cohen \cite{Cohen}, for $1\le k \le n-1$ and $k\neq 2$,
 the automorphism $\omega_k$ is given by the conjugation action of the braid
 $\sigma_k$ on the pure braid group. More precisely,
 $\omega_k(A_{i,j}^{})=A_{i,j}^{\sigma_k}=\sigma_k^{-1}A_{i,j}^{}\sigma_k^{}$.
 On the other hand, $\omega_2$ is the composite of the conjugation action of $\sigma_2$
 and the automorphism $\phi_{1,3}$ given by
$$
\phi_{1,3} \colon A_{i,j} \mapsto  \left\{
\begin{array}{ll}
A_{1,2}z_n&\textrm{if}~i=1, j=2\\
A_{1,3}z_n^{-1}&\textrm{if}~i=1, j=3\\
A_{i,j} &~\textrm{otherwise}.
\end{array} \right.
$$
The automorphism $\omega_n$ is the product of $w_n$ and some central automorphism. Finally, the automorphism $\varepsilon$ is the product of $t$ and some central automorphism.
\end{proof}

We now prove the main result of this section.

\begin{theorem}\label{main-result-pn}
$\Aut(P_n) \cong \Aut_c(P_n) \rtimes G_n$ for $n \ge 4$.
\end{theorem}

\begin{proof}
Obviously, $\Aut_c(P_n) \rtimes G_n \le \Aut(P_n)$. By Bell and Margalit \cite[Theorem 8]{Bell},
$\Aut(P_n)\cong \Aut_c(P_n) \rtimes \Mod(\mathbb{S}_{n+1})$  for $n \ge 4$. Recall that, $\Mod (\mathbb{S}_{n+1})= \langle \omega_1,\dots,\omega_n, \varepsilon \rangle$.
The proof of the theorem is now complete by Lemma \ref{epsilon-lemma}.
\end{proof}

We conclude this section with an explicit presentation of $\Aut(P_4)$. By definition of $\omega_i$'s and $\varepsilon$, we have
$$
\omega_1=s_1,\,\, \omega_2=s_2 \phi_{13},\,\, \omega_3=s_3,\,\, \omega_4=w_4 \phi_{14} \phi_{24},\,\, \varepsilon\tau=\psi.
$$
Now rewriting the relations of $\Aut(P_4)$ in the new generators yields the following result.

\begin{prop}
The group $\Aut(P_4)$ in the generators
$$
\{t,\,\, s_1,\,\, s_2, \,\, s_3, \,\,  w_4, \,\, \psi, \,\phi_{13},\,\, \phi_{23},\,\, \phi_{14},\,\, \phi_{24},\,\, \phi_{34}.\}
$$
is defined by the  relations:
\begin{eqnarray*}
 s_1 s_3 &=& s_3 s_1,\\
  s_1 w_4 &=& w_4 sa_1,\\
 s_2 w_4&=&w_4 s_2(\phi_{13}^{-1} \phi_{24}^{-1} \phi_{34})^2,\\
 s_1 s_2 s_1 &=&s_2 s_1 s_2,\\
 s_3 w_4 s_3 &=&w_4 s_3 w_4 (\phi_{14} \phi_{24} \phi_{34}^{-1})^2,\\
 s_1 s_2 s_3 w_4^2 s_3 s_2 sa_1 &=&(\phi_{13}^{-1} \phi_{23}^{2})^2,\\
 (s_1 s_2 s_3 w_4)^5 &=&(\phi_{13}^{-1} \phi_{23})^3 (\phi_{14}^{-1} \phi_{24}^{-1} \phi_{34})^4,\\
& &\\
 (\psi t)^2&=& \psi^{-1} t^{-1} \psi t=  \psi^2= t^2=1,\\
 \psi s_k&=&s_k \psi ~~\textrm{for}~~ k=1,2,3,\\
 \psi w_4&=&w_4 \psi,\\
 (\psi \phi_{ij})^2&=&1~~\textrm{for all}  ~~\left\{i,j\right\},\\
t \phi_{ij} t &=& \phi_{ij}^{-1} \psi^{-2} ~~\textrm{for all}~ ~ \left\{i,j\right\},\\
\phi_{ij} \phi_{pq} &=& \phi_{pq} \phi_{ij}~~\textrm{for all}~~  \left\{i,j\right\},\left\{p,q \right\},
\end{eqnarray*}

\begin{eqnarray*}
(t s_1)^2&=&\psi^{-2},\\
(t s_2)^2&=&\phi_{13}^{-2},\\
(t s_3)^2&=&\psi^{-2},\\
(t w_4)^2&=&\psi^{2},\\
& &\\
{s_1}^{-1}\phi_{13}{s_1}&=&\phi_{23},\\
{s_1}^{-1}\phi_{23}{s_1}&=&\phi_{13},\\
{s_1}^{-1}\phi_{14}{s_1}&=&\phi_{24},\\
{s_1}^{-1}\phi_{24}{s_1}&=&\phi_{14},\\
{s_1}^{-1}\phi_{34}{s_1}&=&\phi_{34},\\
& &\\
{s_2}^{-1}\phi_{13}{s_2}&=&\phi_{13}^{-1},\\
{s_2}^{-1}\phi_{23}{s_2}&=&\phi_{13}^{-1}\phi_{23},\\
{s_2}^{-1}\phi_{14}{s_2}&=&\phi_{13}^{-1}\phi_{14},\\
{s_2}^{-1}\phi_{24}{s_2}&=&\phi_{13}^{-1}\phi_{34},\\
{s_2}^{-1}\phi_{34}{s_2}&=&\phi_{13}^{-1}\phi_{24},\\
& &\\
{s_3}^{-1}\phi_{13}{s_3}&=&\phi_{14},\\
{s_3}^{-1}\phi_{23}{s_3}&=&\phi_{24},\\
{s_3}^{-1}\phi_{14}{s_3}&=&\phi_{13},\\
{s_3}^{-1}\phi_{24}{s_3}&=&\phi_{23},\\
{s_3}^{-1}\phi_{34}{s_3}&=&\phi_{34},\\
& &\\
{w_4}^{-1}\phi_{13}{w_4}&=&\phi_{13} \phi_{24} \phi_{34}^{-1},\\
{w_4}^{-1}\phi_{23}{w_4}&=&\phi_{23} \phi_{14} \phi_{34}^{-1},\\
{w_4}^{-1}\phi_{14}{w_4}&=&\phi_{24},\\
{w_4}^{-1}\phi_{24}{w_4}&=&\phi_{14},\\
{w_4}^{-1}\phi_{34}{w_4}&=&\phi_{14} \phi_{24}\phi_{34}^{-1}.
\end{eqnarray*}
\end{prop}
\bigskip


\section{Automorphism group of $P_3$}\label{sec4}
For $n=3$, we have
$B_3= \langle\sigma_1, \sigma_2~|~ \sigma_2\sigma_1\sigma_2=\sigma_1\sigma_2\sigma_1\rangle$.
Let $A_{1,2}=\sigma_1^2$, $A_{1,3}=\sigma_2 \sigma_1^2 \sigma_2^{-1}$ and $A_{2,3}=\sigma_2^2$.
Then
$$
P_3=  \langle A_{1,2}, A_{1,3},A_{2,3} \rangle=
 \langle A_{1,2}A_{1,3}A_{2,3} \rangle \times  \langle A_{1,3}, A_{2,3} \rangle.
$$
We set $x= A_{1,3}$, $y=A_{2,3}$ and $z=A_{1,2}A_{1,3}A_{2,3}$.
Then $P_3= \langle z \rangle \times \langle x, y \rangle$, where $\langle x,y \rangle$
is the free group $F_2$ on two generators.

\begin{lemma}\label{action1}
The action of $\sigma_1$ and  $\sigma_2$ on $x,y,z$ is given by
\begin{enumerate}
\item $z^{\sigma_1}=z^{\sigma_2}=z$.
\item $x^{\sigma_1}= xyx^{-1}$ and $x^{\sigma_2}= y^{-1}x^{-1}z$.
\item $y^{\sigma_1}=x$ and $y^{\sigma_2}=y$.
\end{enumerate}
\end{lemma}
\begin{proof}
(1) follows trivially since $z \in Z(B_3)$. For (2), we consider
\begin{eqnarray*}
x^{\sigma_1} &=& A_{1,3}^{\sigma_1}\\
&=& A_{2,3}[A_{1,2}^{-1}, A_{1,3}^{-1}]\\
&=& [(zy^{-1}x^{-1})^{-1}, x^{-1}]\\
&=& y[xy, x^{-1}]=y[x, x^{-1}]^y[y, x^{-1}]\\
&=& y[y, x^{-1}]\\
&=&  xyx^{-1}.
\end{eqnarray*}
The rest of the identities follow easily.
\end{proof}

\begin{remark}
The preceding lemma also implies that the  automorphisms of $P_3$ induced by
$\sigma_1$ and $\sigma_2$ are, in fact, automorphisms of $F_2= \langle x,y \rangle$.
\end{remark}

Next, we describe $\Aut(P_3)$. Consider the following automorphisms of $P_3$
$$
\theta : \left\{
\begin{array}{ll}
x \longmapsto x &  \\
y \longmapsto y &  \\
z \longmapsto z^{-1},&\\

\end{array} \right.
\xi : \left\{
\begin{array}{ll}
x \longmapsto xz &  \\
y \longmapsto y &  \\
z \longmapsto z, &  \\
\end{array} \right.
$$

$$
\eta : \left\{
\begin{array}{ll}
x \longmapsto x &  \\
y \longmapsto yz &  \\
z \longmapsto z,&  \\
\end{array} \right.
\phi : \left\{
\begin{array}{ll}
x \longmapsto x^\psi &  \\
y \longmapsto y^\psi &  \\
z \longmapsto z, &  \\
\end{array} \right.
$$
where $\psi$ is an arbitrary automorphism of $F_2= \langle x,y \rangle$.

\begin{prop}
The group $\Aut(P_3)$ is generated by the set $\{\theta, \xi, \eta,  \phi \}$.
\end{prop}

\begin{proof}
It is enough to prove that any automorphism $f$ of $P_3$ can be
expressed in terms of these four automorphisms.
Since $Z(P_3)= \langle z \rangle$, it follows that $z^f=z^{\pm 1}$.
Further, $f$ induces an automorphism of $P_3/Z(P_3) \cong F_2$.
Hence $x^f \equiv x^{\psi} \mod Z(P_3)$ and $y^f \equiv y^{\psi} \mod Z(P_3)$
for some $\psi \in \Aut(F_2)$. Define $\phi: P_3 \to P_3$ by
$$
\phi : \left\{
\begin{array}{ll}
x \longmapsto x^{\psi} &  \\
y \longmapsto y^{\psi} &  \\
z \longmapsto z. &  \\
\end{array} \right.
$$
Then the automorphism $f  \phi^{-1}$ induce identity on $P_3/Z(P_3)$. Hence
$$
f  \phi^{-1} : \left\{
\begin{array}{ll}
x \longmapsto xz^{\alpha} &  \\
y \longmapsto yz^{\beta} &  \\
z \longmapsto z^{\pm 1} &  \\
\end{array} \right.
$$
for some $\alpha, \beta \in \mathbb{Z}$. Hence $f= \theta \xi^{\alpha} \eta^{\beta} \phi$.
This proves the lemma.
\end{proof}

Let $x \mapsto p(x,y)$ and $y \mapsto q(x,y)$ be an arbitrary automorphism of $F_2= \langle x,y \rangle $.

\begin{lemma}
The group $\Aut(P_3)$ has following defining relations:
\begin{enumerate}
\item $\theta^2=1$
\item $\xi \eta = \eta \xi$
\item $\theta \xi \theta =\xi^{-1}$
\item $\theta \eta \theta =\eta^{-1}$
\item $\theta \phi= \phi \theta$ for all $\phi \in \Aut(F_2)$
\item $\xi \phi \xi^{-1}=\xi^{1-\log_x(p)} \eta^{-\log_x(q)} \phi$
\item $\eta \phi \eta^{-1}=\xi^{-\log_y(p)} \eta^{1-\log_y(q)} \phi$.
\end{enumerate}
\end{lemma}

\begin{proof}
(1)-(5) are evident from the definitions of the automorphisms. For (6), consider
\begin{eqnarray*}
x^{\xi \phi \xi^{-1}} &=& (xz)^{\phi \xi^{-1}}\\
& = & \big(p(x,y)z\big)^{\xi^{-1}}\\
& = & p(xz^{-1}, y)z\\
& = & p(x,y)z^{1-\log_x(p)}\\
& = & x^{\xi^{1-\log_x(p)} \eta^{-\log_x(q)} \phi}.
\end{eqnarray*}

Similarly, for $y$, we obtain
\begin{eqnarray*}
y^{\xi \phi \xi^{-1}} &=& y^{\phi \xi^{-1}}\\
& = & \big(q(x,y)\big)^{\xi^{-1}}\\
& = & q(xz^{-1}, y)\\
& = & q(x,y)z^{1-\log_x(q)}\\
& = & y^{\xi^{1-\log_x(p)} \eta^{-\log_x(q)} \phi}.
\end{eqnarray*}
The identity (7) follows analogously.
\end{proof}

\begin{remark}
Conditions (6) and (7) can be rewritten in the following form:
\begin{enumerate}
\item[(6\')] $\phi \xi \phi^{-1}=\xi^{\log_x(p)} \eta^{\log_x(q)}$.
\item[(7\')] $\phi \eta \phi^{-1}=\xi^{\log_y(p)} \eta^{\log_y(q)}$.
\end{enumerate}
\end{remark}

The above relations yield the following lemma.

\begin{lemma}
$\langle \xi, \eta \rangle $ is a normal subgroup of $\Aut(P_3)$ and
$\Aut(P_3)/\langle \xi, \eta \rangle  \cong \mathbb{Z}_2 \times \Aut(F_2)$.
\end{lemma}

Recall that $P_3 = \mathbb{Z} \times F_2$. In this case,
$\Aut(P_3)= \Aut_c(P_3) \times \Aut(F_2)$, where $\Aut_c(P_3) =\mathbb{Z}^2 \rtimes \mathbb{Z}_2$,
generated by $\psi, \phi_{1,3}, \phi_{2,3}$ with $\psi^2=1$ and
$\psi \phi_{i,j} \psi=  \phi_{i,j}^{-1}$. The generators are given by

$$
\psi :\left\{
\begin{array}{ll}
A_{1,2} \mapsto A_{1,2}z^{-2} &\\
A_{1,3} \mapsto A_{1,3}&\\
A_{2,3} \mapsto A_{2,3},&\\
\end{array} \right.
$$
$$
\phi_{1,3} : \left\{
\begin{array}{ll}
A_{1,2} \mapsto A_{1,2}z&\\
A_{1,3} \mapsto A_{1,3}z^{-1}&\\
A_{2,3} \mapsto A_{2,3}&\\
\end{array} \right.
$$
and
$$
\phi_{2,3} : \left\{
\begin{array}{ll}
A_{1,2} \mapsto A_{1,2}z&\\
A_{1,3} \mapsto A_{1,3}&\\
A_{2,3} \mapsto A_{2,3}z^{-1}.&\\
\end{array} \right.
$$

Note that
$F_2=\overline{P}_3= P_3/Z(P_3)=P_3/\langle A_{1,2}A_{1,3}A_{2,3} \rangle= \langle A_{1,3},A_{2,3} \rangle$.
The group $\Aut(F_2)$ admits the following presentation
$$
\Aut(F_2)=
 \big\langle \rho, \sigma, \nu~|~\rho^2, \sigma^2, (\sigma\rho)^4, (\rho \sigma \rho \nu)^2,
   (\nu \rho \sigma)^3, [\nu, \sigma \nu \sigma] \big\rangle,
$$
where the automorphisms are given by

$$
\rho :\left\{
\begin{array}{ll}
A_{1,3} \mapsto A_{2,3}&\\
A_{2,3} \mapsto A_{1,3},&\\
\end{array} \right.
$$

$$\sigma :\left\{
\begin{array}{ll}
A_{1,3} \mapsto A_{1,3}^{-1}&\\
A_{2,3} \mapsto A_{2,3}&\\
\end{array} \right.
$$
and
$$\nu :\left\{
\begin{array}{ll}
A_{1,3} \mapsto A_{1,3}A_{2,3}&\\
A_{2,3} \mapsto A_{2,3}.&\\
\end{array} \right.
$$

The lifts of these automorphisms to automorphisms of $P_3$ fixing the
$A_{1,2}A_{1,3}A_{2,3}$ are given by setting
$$
\rho(A_{1,2})=A_{2,3}A_{1,3}A_{2,3}^{-1},~\sigma(A_{1,2})=A_{1,2}A_{1,3}^2~ \textrm{and}~\nu(A_{1,2})= A_{2,3}^{-1}A_{1,2}.
$$
 Thus, we have proved the following.

\begin{prop}\label{main-result-p3}
The group $\Aut(P_3)$ is generated by the set $\{\rho, \sigma, \nu, \psi, \phi_{1,3}, \phi_{2,3}\}.$
\end{prop}

\bigskip

\section{Extension and lifting problem for  $1 \to U_n \to P_n \to P_{n-1} \to 1$}\label{sec5}

In this section, we deal the extension and lifting problem for the exact sequence
$$
1 \to U_n \to P_n \to P_{n-1} \to 1.
$$
The cases $n=1, 2$ are vacuous. For $n=3$, we have $P_3 \cong U_3 \times  P_2$, and hence each automorphism of $U_3$ and $P_2$ can be extended to an automorphism of $P_3$. 

We deal the case $n=4$ in the rest of this section. Recall that, by definition, $P_4= \langle A_{1,2}, A_{1,3}, A_{2,3}, A_{1,4}, A_{2,4}, A_{3,4} \rangle$ and $P_3= \langle A_{1,2}, A_{1,3}, A_{2,3} \rangle$. Further, $U_4= \langle A_{1,4}, A_{2,4}, A_{3,4} \rangle$ is normal in $P_4$ and $P_4/U_4 \cong P_3$. The main result of this section is the following theorem.

\begin{theorem}\label{not liftable auto}
There exists an automorphism of $P_3$ which cannot be lifted to an automorphism of $P_4$.
\end{theorem}

The theorem will be proved via the following sequence of lemmas.

\begin{lemma}\label{centraliser}
Let $G$ be a group and $\phi \in \Aut(G)$. For $g \in G$, let
$\C_G(g)= \{x \in G~|~ gx=xg \}$. Then $\C_G(g^{\phi})= \C_G(g)^{\phi}$.
\end{lemma}

\begin{proof}
Let $x \in \C_G(g)$. Then $xg=gx$ implies that $x^\phi g^\phi=g^\phi x^\phi$.
This further implies $x^\phi \in \C_G(g^\phi)$. Hence $\C_G(g)^\phi \subseteq \C_G(g^ \phi)$.
The converse is also obvious.
\end{proof}

\begin{lemma}\label{fix set}
Let $F_n= \langle x_1, \dots, x_n\rangle $, and $\phi \in \Aut(F_n)$ such that
$$
\phi : \left\{
\begin{array}{ll}
x_i \longmapsto x_i & \textrm{for}~ 1 \leq i \leq n-1 \\
x_n \longmapsto w^{-1} x_n w &  \textrm{where }~w \in \langle x_1, \dots, x_{n-1} \rangle. \\
\end{array} \right.
$$
Then $Fix(\phi)=F_{n-1}=\langle x_1, \dots, x_{n-1} \rangle$.
\end{lemma}

\begin{proof}
Let $g \in Fix(\phi)$. Then we can write $g=w_1 x_n^{\alpha_1} \cdots w_m x_n^{\alpha_m} w_{m+1}$
for some $\alpha_1, \dots, \alpha_m \in \mathbb{Z}$ and $w_1, \dots, w_m, w_{m+1} \in F_{n-1}$.
We can assume that $w_{m+1}=1$ and $\alpha_m \neq 0$. Applying $\phi$ gives
\begin{eqnarray*}
w_1 x_n^{\alpha_1} \cdots w_m x_n^{\alpha_m}  &=& g^\phi\\
& = & w_1 (x_n^{\alpha_1})^w \cdots w_m (x_n^{\alpha_m})^w\\
& = & w_1 w^{-1} x_n^{\alpha_1} (w w_2 w^{-1}) \cdots (w w_m w^{-1}) x_n^{\alpha_m} w.\\
\end{eqnarray*}
Note that the elements $w w_i w^{-1} \neq 1$ and we have equality in the free product
$F_n= F_{n-1} \star \langle x_n\rangle$. This implies that $m=1$ and $\alpha_m=0$. Hence $g \in F_{n-1}$.
\end{proof}

We also need the well-known conjugation rules in $P_4$.
Consider the automorphism of $P_3$ of the following form
$$
\phi : \left\{
\begin{array}{ll}
A_{1,2} \longmapsto A_{1,2}A_{1,3}A_{2,3}^{-1}A_{1,3}^{-1} &  \\
A_{1,3} \longmapsto A_{1,3}A_{2,3} &  \\
A_{2,3} \longmapsto A_{2,3}. &
\end{array} \right.
$$

\begin{lemma}
$Fix (\widehat{A_{1,3}})= \langle A_{1,4}A_{3,4}, A_{1,4}A_{2,4}A_{3,4} \rangle $.
\end{lemma}

\begin{proof}
Consider
\begin{eqnarray*}
(A_{1,4}A_{2,4}A_{3,4})^{A_{1,3}}  &=& (A_{1,4}A_{3,4}) A_{1,4} (A_{1,4}A_{3,4})^{-1} (A_{1,4}A_{3,4}A_{1,4}^{-1}A_{3,4}^{-1})A_{2,4}\\
& & (A_{1,4}A_{3,4}A_{1,4}^{-1}A_{3,4}^{-1})^{-1}A_{1,4}A_{3,4}A_{3,4}(A_{1,4}A_{3,4})^{-1}\\
& = & A_{1,4}A_{2,4}A_{3,4}A_{1,4}A_{3,4}^{-1}A_{1,4}^{-1}A_{1,4}A_{3,4}A_{1,4}^{-1}\\
& = & A_{1,4}A_{2,4}A_{3,4}.
\end{eqnarray*}
Clearly, $(A_{1,4}A_{3,4})^{A_{1,3}}=A_{1,4}A_{3,4}$. Note that 
$$U_4= \langle A_{1,4},A_{2,4}, A_{3,4} \rangle= \langle A_{1,4}A_{2,4}A_{3,4}, A_{1,4}A_{3,4}, A_{1,4}\rangle.$$
 In these generators, we have
 \begin{eqnarray*}
(A_{1,4}A_{2,4}A_{3,4})^{A_{1,3}}&=& A_{1,4}A_{2,4}A_{3,4},\\
(A_{1,4}A_{3,4})^{A_{1,3}}&=& A_{1,4}A_{3,4},\\
A_{1,4}^{A_{1,3}}&=& A_{1,4}A_{3,4}A_{1,4} (A_{1,4}A_{3,4})^{-1}.
\end{eqnarray*}
From Lemma \ref{fix set}, we obtain $Fix (A_{1,3})= \langle A_{1,4}A_{3,4}, A_{1,4}A_{2,4}A_{3,4} \rangle $.
\end{proof}

\begin{lemma}
$Fix(\widehat{A_{1,3}A_{2,3}})=\langle  A_{1,4}A_{2,4}A_{3,4} \rangle$.
\end{lemma}
\begin{proof}
First, note that $A_{1,4}A_{2,4}A_{3,4} \in Fix (A_{1,3}A_{2,3})$. Now, we determine action of $A_{1,3}A_{2,3}$ on $U_4$. We have
\begin{eqnarray*}
A_{1,4}^{A_{1,3}A_{2,3}}  &=& \big( (A_{1,4}A_{3,4})A_{1,4}(A_{1,4}A_{3,4})^{-1} \big)^{A_{2,3}}\\
& = & (A_{1,4}A_{2,4}A_{3,4}A_{2,4}^{-1}) A_{1,4} (A_{1,4}A_{2,4}A_{3,4}A_{2,4}^{-1})^{-1}\\
&= & (A_{1,4}A_{2,4}A_{3,4}) A_{2,4}^{-1} A_{1,4} A_{2,4} (A_{1,4}A_{2,4}A_{3,4})^{-1}
\end{eqnarray*}
and
\begin{eqnarray*}
A_{3,4}^{A_{1,3}A_{2,3}}  &=& (A_{1,4}A_{3,4}A_{1,4}^{-1})^{A_{2,3}}\\
& = & (A_{1,4}A_{2,4}A_{3,4}) A_{3,4} (A_{1,4}A_{2,4}A_{3,4})^{-1}.
\end{eqnarray*}
Also, we have $(A_{1,4}A_{2,4}A_{3,4})^{A_{1,3}A_{2,3}}= A_{1,4}A_{2,4}A_{3,4}$. If we set $x=A_{1,4}A_{2,4}A_{3,4}$, $y=A_{3,4}$ and $z=A_{1,4}$, then
\begin{eqnarray*}
x^{A_{1,3}A_{2,3}}&=& x,\\
y^{A_{1,3}A_{2,3}}&=& xy x^{-1},\\
z^{A_{1,3}A_{2,3}}&=&(xyz^{-1})z (xyz^{-1})^{-1}=(xy)z(xy)^{-1}.
\end{eqnarray*}
Let $w \in Fix(A_{1,3}A_{2,3})$. Then $w=w_1 z^{\alpha_1} w_2 z^{\alpha_2} \cdots w_m z^{\alpha_m}w_{m+1}$, where $w_1, w_2, \dots, w_{m+1} \in \langle y, x\rangle $. Now
\begin{eqnarray*}
w &=& w^{A_{1,3}A_{2,3}} \\
&=& (x w_1 x^{-1}) (xy z^{-1}) z^{\alpha_1} (xy z^{-1})^{-1}\cdots (xy z^{-1}) z^{\alpha_m}(xy z^{-1})^{-1} xw_m x^{-1}\\
& = & w_1^x (xy)z^{\alpha_1} (xy)^{-1} \cdots (xy)z^{\alpha_m} (xy)^{-1} w_m^x\\
& = & (w_1^x xy) z^{\alpha_1} \big( (xy)^{-1} w_2^x (xy) \big) \cdots \big( (xy)^{-1} w_{m-1}^x (xy) \big) z^{\alpha_m} \big((xy)^{-1} w_m^x \big).
\end{eqnarray*}
This implies $(xy)^{-1}w_m^x= w_{m+1}$, which is a contradiction if $\alpha_m \neq 0$. Hence $\alpha_1=\cdots =\alpha_m =0$ and $w=x^{\alpha_1} y^{\beta_1} \cdots x^{\alpha_n} y^{\beta_n}$. Now, by Lemma \ref{fix set}, $w=  x^{\alpha}$ for some $\alpha$. This completes the proof of the lemma.
\end{proof}

\textbf{Proof of Theorem \ref{not liftable auto}. }
If we lift the automorphism $\phi$ to $P_4$, then $Fix(\widehat{A_{1,3}}) \cong F_2$
and $Fix(\widehat{A_{1,3}A_{2,3}}) \cong \mathbb{Z}$, which is a contradiction.
Hence $\phi$ cannot be lifted to $P_4$. $\Box$

In the reverse direction, we have the following.

\begin{theorem}\label{not-extendable-auto}
There exists an automorphism of $U_4$ which cannot be extended to an automorphism of $P_4$.
\end{theorem}

\begin{proof}
Consider the automorphism of $U_4$ of the following form
$$
\phi : \left\{
\begin{array}{ll}
A_{1,4} \longmapsto A_{1,4}A_{2,4} &  \\
A_{2,4} \longmapsto A_{2,4} &  \\
A_{3,4} \longmapsto A_{3,4}. &  \\
\end{array} \right.
$$
Suppose that we can extend $\phi$ to an automorphism of $P_4$.
Then $\phi$ induces an automorphism of $P_3$.
Recall that $Z(P_3)= \langle A_{1,2}A_{1,3}A_{2,3}\rangle$.
Then $(A_{1,2}A_{1,3}A_{2,3})^\phi =  (A_{1,2}A_{1,3}A_{2,3})^{\pm 1}$.
Also, $A_{1,2}A_{1,3}A_{2,3}$ acts on $U_4$ by conjugating $A_{1,4}A_{2,4}A_{3,4}$.
Hence $(A_{1,4}A_{2,4}A_{3,4})^{A_{1,2}A_{1,3}A_{2,3}}= A_{1,4}A_{2,4}A_{3,4}$.
Applying $\phi$, we get
$$
\big((A_{1,4}A_{2,4}A_{3,4})^{\phi} \big)^{(A_{1,2}A_{1,3}A_{2,3})^\phi}= (A_{1,4}A_{2,4}A_{3,4})^\phi.
$$
This further implies
$(A_{1,4}A_{2,4}^2A_{3,4})^{(A_{1,4}A_{2,4}A_{3,4})^{\pm 1}}= A_{1,4}A_{2,4}^2A_{3,4}$,
which is a contradiction. Hence $\phi$ cannot be extended to an automorphism of $P_4$.
\end{proof}

\begin{question}
What can we say about the lifting and extension problem for the exact sequence
$$
1 \to U_n \to P_n \to P_{n-1} \to 1
$$
for $n \ge 5$.
\end{question}

\bigskip

\section{Extension and lifting problem for $1 \to P_n \to B_n \to S_n \to 1$}\label{sec6}

In this section, we investigate the extension and lifting problem for the well-known short exact sequence
$$1 \to P_n \to B_n \stackrel{\pi}{\to} S_n \to 1.$$

The following is a straightforward observation.

\begin{lemma}\label{extend-inner}
Let $1 \to K \to G \to H \to 1$ be a short exact sequence of groups. Then every $\phi \in \Inn(H)$
has a lift in $ \Aut(G)$ and every $\psi \in \Inn(K)$ has an extension in $\Aut(G)$.
\end{lemma}

We now prove the following.

\begin{prop}\label{lift-inner-sn}
If $n \neq 6$, then any automorphism $\phi \in \Aut(S_n)$ can be lifted to an
automorphism of $B_n$. Further, the non-inner automorphism of $S_6$ cannot be
lifted to an automorphism of $B_6$.
\end{prop}

\begin{proof}
We know that if $n \neq 2$ or 6, then $\Aut(S_n)= \Inn(S_n)$,
and hence the result follows from Lemma \ref{extend-inner}. For $n=2$,
we have $B_2= \langle \sigma_1\rangle$, $P_2= \langle \sigma_1^2\rangle$  and
$S_2=\mathbb{Z}_2$. Note that  $\Aut(S_2)=1$, and the identity automorphism is obviously liftable to $B_2$.
On the other hand, $\Aut(P_2)= \mathbb{Z}_2$, say, generated  by $\psi$.
Then $\psi(\sigma_1^2)=\sigma_1^{-2}$.
Define $\psi^{\circ}: B_2 \to B_2$ by $\psi^{\circ}(\sigma_1)=\sigma_1^{-1}$.
Then $\psi^{\circ}$ is an extension of $\psi$.

Finally, let $n=6$ and $\phi \in \Aut(S_6)\setminus \Inn(S_6)$.
Since $\Aut (S_6)=\Inn(S_6)\rtimes \mathbb{Z}_2$, we have $\phi^2=1$.
If $\tilde{\phi}$ is a lift of $\phi$, then $\tilde{\phi}=\hat{g}$
for some $g \in B_6$ or $\tilde{\phi}=\hat{g} \tau$, where
$\tau(\sigma_i)=\sigma_i^{-1}$. Here $\hat{g}$ is the inner automorphism of $B_6$ induced by $g$. If
$\tilde{\phi}=\hat{g}$, then applying on the generators yield
$$s_i^\phi=(\sigma_i^\pi)^\phi= \big(\sigma_i^{\tilde{\phi}}\big)^\pi=(g^\pi)^{-1} s_i g^\pi$$
for all $1 \leq i \leq n-1$. This implies that $\phi=\widehat{g^\pi}$, a contradiction.

Now suppose that $\tilde{\phi}=\hat{g} \tau$.
Then
$$s_i^\phi=(\sigma_i^\pi)^\phi  =  \big(\sigma_i^{\tilde{\phi}}\big)^\pi=\big(\tau(g)^{-1}\sigma_i^{-1} \tau(g) \big)^\pi=      (\tau(g)^\pi)^{-1}s_i \tau(g)^\pi$$ for all $1 \leq i \leq n-1$.
This implies that $\phi=\widehat{\tau(g)^\pi}$, which is again a contradiction.
\end{proof}

Recall that $P_n \cong Z(P_n) \times P_n/Z(P_n)$, where $Z(P_n)=\langle z_n \rangle $ is infinite cyclic. Further, the group of central automorphisms of $P_n$ is given by
$$\Aut_c(P_n) \cong tv^\circ(P_n) \rtimes \mathbb{Z}_2,$$
 where $tv^\circ(P_n) = \mathbb{Z}^N$ with $N= \binom{n}{2}-1$ and $\mathbb{Z}_2=\langle \theta_0\rangle $.
Here  $z_n^{\theta_0}=z_n^{-1}$ and it acts trivially on $P_n/Z(P_n)$.

\begin{prop}\label{autcent-circ}
No non-trivial element of $tv^\circ(P_n)$ can be extended to an automorphism of $B_n$.
\end{prop}

\begin{proof}
First observe that
$$\Aut(B_n)/\Inn(P_n) \cong \langle \Inn(B_n), \tau \rangle/\Inn(P_n)\cong \langle S_n, \tau\rangle$$
is a finite group. Let $\phi^\circ \in tv^\circ(P_n)$ be extendable to an automorphism $\phi$ of $B_n$.
Then there exists an integer $m$ such that $\phi^m \in \Inn(P_n)$,
and hence $\phi^m|_{P_n}={\phi^\circ}^{m}|_{P_n}$.
This implies $\phi^m \in tv^\circ(P_n)\cap \Inn(P_n)= \{1\}$,
and hence ${\phi^\circ}^{m}=1$. But the group $tv^\circ(P_n)$ is free abelian,
and hence does not have any non-trivial element of finite order. Therefore $\phi^\circ=1$.
\end{proof}

\begin{lemma}
Let $\tau \in \Aut(B_n)$ be the automorphism given by $\tau(\sigma_i)=\sigma_i^{-1}$
for  $1 \leq i \leq n-1$. Then $z_n^\tau=z_n^{-1}$.
\end{lemma}

\begin{proof}
Recall that $z_n=(\sigma_1\cdots \sigma_{n-1})^n$. Thus $z_n^\tau=(\sigma_1^{-1}\cdots \sigma_{n-1}^{-1})^n$.
Since $Z(B_n)= \langle z_n \rangle \cong \mathbb{Z}$, we have $z_n^\tau= z_n^{\pm1}$.
Note that $B_n/B_n'$ is infinite cyclic generated by $\sigma_1 B_n'$.
If  $z_n^\tau= z_n$, then reading this equation modulo $B_n'$ gives
$\sigma_1^{(n-1)n} \equiv \sigma_1^{-(n-1)n} \mod B_n'$.
Equivalently, $\sigma_1^{2(n-1)n} \equiv 1 \mod B_n'$,
which is a contradiction as  $B_n/B_n'$ has no element of finite order. Hence $z_n^\tau=z_n^{-1}$.
\end{proof}

\begin{prop}\label{theta0}
$\theta_0$ cannot be extended to an automorphism of $B_n$.
\end{prop}

\begin{proof}
Suppose that $\theta \in \Aut(B_n)$ is an extension of $\theta_0$.
Then $z_n^\theta=z_n^{\theta_0}=z_n^{-1}$. Recall that
$\Aut(B_n) \cong  \Inn(B_n) \rtimes \langle \tau \rangle$.
Since for each $\psi \in \Inn(B_n)$, we have  $z_n^\psi=z_n$, it follows that
$\theta=\tau \psi$ for some $\psi \in \Inn(B_n)$. It follows from the relations in
$B_n$ that $\psi \in \Inn(B_n)$ induces a permutation of $P_n/P_n'$.
On the other hand, $\tau$ induces inversion in $P_n/P_n'$, more precisely,
$A_{i,j}^{\tau}\equiv A_{i,j}^{-1} \mod P_n'$ for all $i <j$. We can assume that
$$
P_n/Z(P_n)=
 \langle A_{1,3},A_{2,3}, A_{1,4}, A_{2,4}, A_{3,4}, \dots, A_{1, n}, A_{2,n}, \dots, A_{n-1,n} \rangle.
$$
Then we have
$$A_{1,3}=A_{1,3}^\theta=A_{1,3}^{\tau \psi} \equiv (A_{1,3}^{-1})^\psi \equiv A_{i,j}^{-1}\mod P_n',$$
for some $1 \leq i < j \leq n$, which is a contradiction. This proves the proposition.
\end{proof}

Finally, we prove the main result of this section.

\begin{theorem}\label{main-result-sn}
No non-trivial element of $\Aut_c(P_n)$ can be extended to an automorphism of $B_n$.
\end{theorem}

\begin{proof}
Let $\phi=\alpha \tau^\epsilon \in \Aut_c(P_n)$, where $\alpha \in tv^\circ(P_n)$ and $\epsilon=0, 1$.
In view of Propositions \ref{autcent-circ} and \ref{theta0},
we can assume that $\alpha \neq 1$ and $\epsilon=1$.
Note that $\phi^2=(\alpha \tau)^2=\alpha^2 \neq 1$ since  $tv^\circ(P_n)$ is free abelian.
By Proposition \ref{autcent-circ}, $\alpha^2$ cannot be extended to an automorphism of $B_n$.
Therefore $\phi^2$, and hence $\phi$ cannot be extended to an automorphism of $B_n$.
\end{proof}

The following question remains.
\begin{question}
Which non-central automorphisms of $P_n$ can be extended to automorphisms of $B_n$?
\end{question}

\begin{ack}
The authors gratefully acknowledge support from the Russian Science Foundation Project No.16-41-02006 and the DST-RSF Project INT/RUS/RSF/P-2.
\end{ack}

\end{document}